\documentclass[11pt,twoside]{amsart}
\usepackage{latexsym,amssymb,amsmath}
\numberwithin{equation}{section}
\newtheorem{theorem}{Theorem}[section]
\newtheorem{lemma}[theorem]{Lemma}
\newtheorem{proposition}[theorem]{Proposition}
\newtheorem{corollary}[theorem]{Corollary}

\theoremstyle{definition}
\newtheorem{definition}[theorem]{Definition}

\newtheorem{remark}[theorem]{Remark}

\begin{document}



\title[Binomial vanishing ideals]{Binomial vanishing ideals}

\author{Azucena Tochimani}
\address{
Departamento de
Matem\'aticas\\
Centro de Investigaci\'on y de Estudios
Avanzados del
IPN\\
Apartado Postal
14--740 \\
07000 Mexico City, D.F.
}
\email{tochimani@math.cinvestav.mx}
\thanks{The first author was partially supported by CONACyT. The
second author was partially supported by SNI}

\author{Rafael H. Villarreal}
\address{
Departamento de
Matem\'aticas\\
Centro de Investigaci\'on y de Estudios
Avanzados del
IPN\\
Apartado Postal
14--740 \\
07000 Mexico City, D.F.
}
\email{vila@math.cinvestav.mx}


\begin{abstract} 
In this note we characterize, in algebraic and geometric terms, when
a graded vanishing ideal is 
generated by binomials over any field $K$.
\end{abstract}

\maketitle

\section{Introduction}\label{section-intro}

Let $S=K[t_1,\ldots,t_s]$ be a polynomial ring
over a field $K$ with the standard grading induced by setting
$\deg(t_i)=1$  for all $i$. By the {\it dimension\/} of an ideal $I\subset S$ 
we mean the Krull dimension of $S/I$.  The affine and projective spaces
over the field $K$ of dimensions $s$ and $s-1$ are denoted by
$\mathbb{A}^s$ and $\mathbb{P}^{s-1}$, respectively. Points of ${\mathbb P}^{s-1}$ are denoted by $[\alpha]$, where $0\neq \alpha\in
\mathbb{A}^s$. 

Given a set $\mathbb{Y}\subset \mathbb{P}^{s-1}$ define 
$I(\mathbb{Y})$, the {\it vanishing ideal\/} of $\mathbb{Y}$, 
as the graded ideal generated by the homogeneous polynomials 
in $S$ that vanish at all points of $\mathbb{Y}$. Conversely, 
given a homogeneous ideal $I\subset S$ 
define $V(I)$, the {\it zero set\/} of $I$, as the set of all 
$[\alpha]\in \mathbb{P}^{s-1}$ such that $f(\alpha)=0$ for all 
homogeneous polynomial $f\in I$. The zero sets are the closed sets of
the {\it Zariski topology\/} of $\mathbb{P}^{s-1}$. The Zariski closure
of $\mathbb{Y}$ is denoted by $\overline{\mathbb{Y}}$. 

We will use the following multi-index notation: for
$a=(a_1,\ldots,a_s)\in\mathbb{Z}^s$, set $t^a=t_1^{a_1}\cdots
t_s^{a_s}$. We call $t^a$ a {\it Laurent monomial\/}.  
If $a_i\geq 0$ for all $i$, $t^a$ is called a {\it monomial} of $S$. 
A {\it binomial\/} of $S$ is an
element of the form $f=t^a-t^b$, for some 
$a,b$ in $\mathbb{N}^s$. An ideal $I\subset S$ generated by 
binomials is called a {\it binomial ideal\/}. A binomial ideal
$I\subset S$ with the property that $t_i$ is not a zero-divisor 
of $S/I$ for all $i$ is called a {\it lattice ideal\/}. 

In this note we classify binomial vanishing ideals 
in algebraic and geometric terms. There are some reasons to study
vanishing ideals. 
They are used in algebraic geometry \cite{harris} and
algebraic coding theory \cite{GRT,cartesian-codes}. They are also used in 
polynomial interpolation problems
\cite{gasca-sauer,interpolation-finite,vanishing-ideals}. 

The set $\mathcal{S}=\mathbb{P}^{s-1}\cup\{[0]\}$ is a monoid under componentwise
multiplication, that is, given $[\alpha]=[(\alpha_1,\ldots,\alpha_s)]$
and $[\beta]=[(\beta_1,\ldots,\beta_s)]$ in $\mathcal{S}$, the product
operation is given by 
$$
[\alpha]\cdot[\beta]=[\alpha\cdot\beta]=[(\alpha_1\beta_1,\ldots,\alpha_s\beta_s)],
$$
where $[\mathbf{1}]=[(1,\ldots,1)]$ is the identity element.
Accordingly the affine space $\mathbb{A}^s$ is also a monoid under componentwise
multiplication.

The contents of this note are as follows. In 
Section~\ref{section-prelim} we recall some preliminaries 
on projective varieties and vanishing ideals. Let $\mathbb{Y}$ be a subset of
$\mathbb{P}^{s-1}$. If $\mathbb{Y}\cup\{[0]\}$ is a submonoid of
$\mathbb{P}^{s-1}\cup\{[0]\}$, we show that $I(\mathbb{Y})$ is a binomial
ideal (Theorem~\ref{monoid-binomial}). The same type of result
holds if $Y$ is a subset of $\mathbb{A}^s$
(Proposition~\ref{monoid-binomial-affine}). 
Then we show that $I(\mathbb{Y})$ is a
binomial ideal if and only if $V(I(\mathbb{Y}))\cup\{[0]\}$ is a 
monoid under componentwise multiplication
(Theorem~\ref{classification-binomial}). As a result if
$\mathbb{Y}$ is finite, 
then  $I(\mathbb{Y})$ is a binomial ideal if and
only if $\mathbb{Y}\cup\{0\}$ is a monoid
(Corollary~\ref{oct11-14-1}). This essentially classifies all 
graded binomial vanishing ideals of dimension $1$ (Corollary~\ref{oct11-14-2})

If $Y$ is a submonoid of an affine torus (see Definition~\ref{projectivetorus-def}), 
then $I(Y)$ is a non-graded lattice ideal
\cite[Proposition~2.3]{EisStu}. We give a graded version of this
result, namely, if $\mathbb{Y}$ is a submonoid of a projective torus,
then $I(\mathbb{Y})$ is a lattice ideal (Corollary~\ref{oct11-14}).

Let $I(\mathbb{Y})$ be a vanishing ideal of dimension $1$. According
to \cite[Proposition~6.7(a)]{lattice-dim1} 
$I(\mathbb{Y})$ is a 
lattice ideal if and only if $\mathbb{Y}$ is a finite subgroup of a
projective torus. We complement this
result by showing that---over an algebraically closed field---$\mathbb{Y}$ is a finite subgroup of a
projective torus if and only if there is a finite subgroup $H$ of
$K^*=K\setminus\{0\}$ and $v_1,\ldots,v_s\in\mathbb{Z}^n$ 
that parameterize $\mathbb{Y}$
relative to $H$ (Proposition~\ref{lattice-ideal-param-h}). For finite
fields, this result was shown in 
\cite[Proposition~6.7(b)]{lattice-dim1}. 

Finally, we classify the graded lattice ideals of dimension $1$ over 
an algebraically closed field of characteristic zero. It turns out
that they are  the vanishing ideals of finite subgroups of projective
tori (Proposition~\ref{lattice-alg-closed-char=0}).

For all unexplained
terminology and additional information,  we refer to 
\cite{CLO,harris} (for algebraic geometry and vanishing ideals) 
and \cite{EisStu,prim-dec-critical,monalg-rev} (for binomial and lattice ideals).

\section{Preliminaries}\label{section-prelim}
In this section, we 
present a few results that will be needed in this note.
All results of this
section are well-known. 

\begin{definition}
Let $K$ be a field. We define 
the {\it projective space\/} of 
dimension $s-1$
over $K$, denoted by 
$\mathbb{P}_K^{s-1}$ or $\mathbb{P}^{s-1}$ if $K$ is understood, to be the quotient space 
$$(K^{s}\setminus\{0\})/\sim $$
where two points $\alpha$, $\beta$ in $K^{s}\setminus\{0\}$ 
are equivalent under $\sim$ if $\alpha=c\beta$ for some $c\in K$. It is usual to denote the 
equivalence class of $\alpha$ by $[\alpha]$. The {\it affine space\/}
of dimension $s$ over the field $K$, denoted $\mathbb{A}_K^s$ or
$\mathbb{A}^s$, is $K^s$.
\end{definition}

For any set $\mathbb{Y}\subset \mathbb{P}^{s-1}$ define 
$I(\mathbb{Y})$, the {\it vanishing ideal\/} of $\mathbb{Y}$, 
as the ideal generated by the homogeneous polynomials 
in $S$ that vanish at all points of $\mathbb{Y}$. Conversely, 
given a homogeneous ideal $I\subset S$ 
define its {\it zero set\/} as  
$$V(I)=\left\{[\alpha]\in \mathbb{P}^{s-1}\vert\, 
f(\alpha)=0,\, 
\forall f\in I\, \mbox{ homogeneous} \right\}.
$$ 

A {\it projective variety\/} is the zero set of a 
homogeneous ideal. It is not difficult to see that the members 
of the family 
$$\tau=\{\mathbb{P}^{s-1}\setminus 
V(I)\, \vert\, I\mbox{ is a graded ideal of }S\}$$
are the open sets of a topology on $\mathbb{P}^{s-1}$, called the 
{\it Zariski topology\/}. In a similar way we can define affine
varieties, vanishing ideals of subsets of the affine space $\mathbb{A}^s$, and the
corresponding Zariski topology of $\mathbb{A}^s$. The Zariski closure
of $\mathbb{Y}$ is denoted by $\overline{\mathbb{Y}}$.

\begin{lemma}\label{mar31-14}
Let $K$ be a field. 

{\rm(a)} \cite[pp.~191--192]{CLO} If $Y\subset\mathbb{A}^s$ and
$\mathbb{Y}\subset\mathbb{P}^{s-1}$,  
then $\overline{Y}=V(I(Y))$ and
$\overline{\mathbb{Y}}=V(I(\mathbb{Y}))$. 

{\rm(b)} If $K$ is a finite field, then $Y=V(I(Y))$ and
$\mathbb{Y}=V(I(\mathbb{Y}))$. 
\end{lemma}

\begin{proof} Part (b) follows from (a) because
$\overline{Y}=Y$ and
$\overline{\mathbb{Y}}=\mathbb{Y}$, if $K$ is finite.
 \end{proof}

\begin{lemma}\label{jun3-14} Let $K$ be a field. If $Y$ is a subset of
$\mathbb{A}^s$ or a subset of $\mathbb{P}^{s-1}$ and $Z=V(I(Y))$,
then $I(Z)=I(Y)$. In particular $I(Y)=I(\overline{Y})$.
\end{lemma}

\begin{proof} Since $Y\subset Z$, we get 
$I(Z)\subset I(Y)$. As $I(Z)=I(V(I(Y)))\supset I(Y)$,
one has equality.
\end{proof}

\begin{lemma}{\cite[Proposition~6, p.~441]{CLO}}\label{dim=1-finite} 
If $\mathbb{Y}\subset\mathbb{P}^{s-1}$ and
$\dim(S/I(\mathbb{Y}))=1$, then $|\mathbb{Y}|<\infty$. 
\end{lemma}


The converse of Lemma~\ref{dim=1-finite} is true. This follows from
the next result.  

\begin{lemma}\label{primdec-ixx} Let $\mathbb{Y}$ and $Y$ be finite
subsets of $\mathbb{P}^{s-1}$ and $\mathbb{A}^s$ respectively, 
let $P$ and $[P]$ be points in $\mathbb{Y}$ and $Y$, respectively, 
with $P=(\alpha_1,\ldots,\alpha_s)$, and let $I_{[P]}$ and $I_P$ be the vanishing ideal of $[P]$ and $P$,
respectively. Then
\begin{equation}
I_{[P]}=(\left\{\alpha_kt_i-\alpha_it_k\vert\, k\neq
i\in\{1,\ldots,s\}\right\}),\ \,
I_{P}=(t_1-\alpha_1,\ldots,t_s-\alpha_s),
\end{equation}
where $\alpha_k\neq 0$ for some $k$. Furthermore
$I(\mathbb{Y})=\bigcap_{[Q]\in \mathbb{Y}}I_{[Q]}$, 
$I(Y)=\bigcap_{Q\in Y}I_{Q}$, $I_{[P]}$ is a prime ideal of height
$s-1$ and $I_P$ is a prime ideal of height $s$.  
\end{lemma}

\section{A classification of vanishing ideals generated by binomials}
We continue to employ the notations and 
definitions used in Sections~\ref{section-intro} and
\ref{section-prelim}. In this part we classify vanishing ideals
generated by binomials. 

Let $(\mathcal{S},\,\cdot\,,1)$ be a monoid and let $K$ be a field. As usual we
define a {\it character\/} $\chi$ of $\mathcal{S}$ in $K$ 
(or a $K$-{\it character} of $\mathcal{S}$) to be a homomorphism of
$\mathcal{S}$ into the multiplicative monoid $(K,\,\cdot,\,1)$. Thus
$\chi$ is a map of $\mathcal{S}$ into $K$ such that $\chi(1)=1$ and
$\chi(\alpha\beta)=\chi(\alpha)\chi(\beta)$ for all $\alpha,\beta$ in
$\mathcal{S}$.

\begin{theorem}{\rm(Dedekind's Theorem \cite[p.~291]{JacI})}\label{dedekind-independence}
If $\chi_1,\ldots,\chi_m$ are distinct characters of a monoid
$\mathcal{S}$ 
into a field $K$, then the only elements $\lambda_1,\ldots,\lambda_m$
in $K$ such that 
$$
\lambda_1\chi_1(\alpha)+\cdots+\lambda_m\chi_m(\alpha)=0
$$
for all $\alpha\in\mathcal{S}$ are $\lambda_1=\cdots=\lambda_m=0$.
\end{theorem}

\begin{theorem}\label{monoid-binomial} If $\mathbb{Y}$ is a subset of
$\mathbb{P}^{s-1}$ and $\mathbb{Y}\cup\{[0]\}$ is a submonoid of
$\mathbb{P}^{s-1}\cup\{[0]\}$ under componentwise multiplication, then $I(\mathbb{Y})$ is a binomial
ideal.  
\end{theorem}

\begin{proof} The set $\mathcal{S}=\{x\in\mathbb{A}^s\,\vert\, 
[x]\in\mathbb{Y}\cup\{[0]\}\}$ is a submonoid of $\mathbb{A}^s$. Take
a homogeneous polynomial $0\neq
f=\lambda_1t^{a_1}+\cdots+\lambda_mt^{a_m}$ that vanishes at all points
of $\mathbb{Y}$, where $\lambda_i\in K\setminus\{0\}$ for
all $i$ and $a_1,\ldots,a_m$ are distinct non-zero vectors in
$\mathbb{N}^s$. We set $a_i=(a_{i_1},\ldots,a_{is})$ for
all $i$. For each $i$ consider the $K$-character 
of $\mathcal{S}$ given by 
$$\chi_{i}\colon\mathcal{S}\rightarrow K,\ \ \ 
(\alpha_1,\ldots,\alpha_s)\mapsto\alpha_1^{a_{i1}}\cdots
\alpha_s^{a_{is}}.
$$ 

As $f\in I(\mathbb{Y})$, one has that
$\lambda_1\chi_1+\cdots+\lambda_m\chi_m=0$. Hence, by
Theorem~\ref{dedekind-independence}, we get that 
$m\geq 2$ and $\chi_i=\chi_j$ for
some $i\neq j$. Thus $t^{a_i}-t^{a_j}$ is in $I(\mathbb{Y})$. 
For simplicity of notation we assume that $i=1$
and $j=2$. Since $\mathbf[\mathbf{1}]\in\mathbb{Y}$, we get that
$\lambda_1+\cdots+\lambda_m=0$. Thus
$$
f=\lambda_2(t^{a_2}-t^{a_1})+\cdots+\lambda_m(t^{a_m}-t^{a_1}).
$$ 

Since $f-\lambda_2(t^{a_2}-t^{a_1})$ is a homogeneous polynomial in
$I(\mathbb{Y})$, by 
induction on $m$, we obtain that $f$ is a sum of homogeneous binomials in
$I(\mathbb{Y})$.
\end{proof}

This result can be restated as:

\begin{theorem} Let $\mathbb{Y}$ be a subset of $\mathbb{P}^{s-1}$ such
that $[\mathbf{1}]\in \mathbb{Y}$ and $[\alpha]\cdot[\beta]\in\mathbb{Y}$
for all $[\alpha]$, $[\beta]$ in $\mathbb{Y}$ with
$\alpha\cdot\beta\neq 0$. Then $I(\mathbb{Y})$ is a binomial ideal.  
\end{theorem}


The next result was observed in the Remark after \cite[Proposition~2.3]{EisStu}.

\begin{proposition}{\rm\cite{EisStu}}\label{monoid-binomial-affine} 
If $Y$ is a submonoid of $\mathbb{A}^s$ and
$\tau\in K^*$, then $I(Y)$ is a binomial ideal and $I(\tau Y)$ is a
non-pure binomial ideal.
 \end{proposition}

\begin{proof} That $I(Y)$ is a binomial ideal follows readily by adapting the
proof of Theorem~\ref{monoid-binomial}. Let
$\{t^{b_i}-t^{c_i}\}_{i=1}^r$ be a set of generators of $I(Y)$ with
$b_i,c_i$ in $\mathbb{N}^s$ for all $i$. If
$a=(a_1,\ldots,a_s)\in\mathbb{N}^s$, we set $|a|=\sum_ia_i$. Then
it is not hard to see that the set
$\{t^{b_i}/\tau^{|b_i|}-t^{c_i}/\tau^{|c_i|}\}_{i=1}^r$ generates
$I(\tau Y)$, that is, $I(\tau Y)$ is a non-pure binomial ideal.
\end{proof}

\begin{theorem}\label{classification-binomial} Let $K$ be a field and
let $\mathbb{Y}$ be 
a subset of $\mathbb{P}^{s-1}$. Then $I(\mathbb{Y})$ is a
binomial ideal if and only if $V(I(\mathbb{Y}))\cup\{[0]\}$ is a 
monoid under componentwise multiplication. 
\end{theorem}

\begin{proof} $\Rightarrow$) Consider an arbitrary non-zero binomial
$f=t^a-t^b$ in $I(\mathbb{Y})$ with $a=(a_i)$ and
$b=(b_i)$ in $\mathbb{N}^s$. As $I(\mathbb{Y})$ is graded, $f$ is
homogeneous. First notice that $[\mathbf{1}]\in V(I(\mathbb{Y}))$
because $f$ vanishes at $[\mathbf{1}]$. Take $[\alpha]$, $[\beta]$ in
$V(I(\mathbb{Y}))$ with $ \alpha=(\alpha_i)$, $\beta=(\beta_i)$. Then
$$
\alpha_1^{a_1}\cdots\alpha_s^{a_s}=\alpha_1^{b_1}\cdots\alpha_s^{b_s}\mbox{
and
}\beta_1^{a_1}\cdots\beta_s^{a_s}=\beta_1^{b_1}\cdots\beta_s^{b_s},
$$
and consequently $(\alpha_1\beta_1)^{a_1}\cdots(\alpha_s\beta_s)^{a_s}=
(\alpha_1\beta_1)^{b_1}\cdots(\alpha_s\beta_s)^{b_s}$, i.e., $f$
vanishes at $[\alpha]\cdot[\beta]=[\alpha\cdot\beta]$ if $\alpha\cdot\beta\neq 0$.
Thus $[\alpha]\cdot[\beta]\in V(I(\mathbb{Y}))\cup\{[0]\}$.

$\Leftarrow$) Thanks to Theorem~\ref{monoid-binomial} one has that
$I(V(I(\mathbb{Y})))$ is a binomial ideal. 
Recall that $V(I(\mathbb{Y}))$ is equal to $\overline{\mathbb{Y}}$
(see Lemma~\ref{mar31-14}). On the other
hand, by Lemma~\ref{jun3-14},
$I(\mathbb{Y})=I(\overline{\mathbb{Y}})$. Thus $I(\mathbb{Y})$ is a
binomial ideal.
\end{proof}

\begin{remark}\label{classification-binomial-affine} If $Y\subset\mathbb{A}^s$, then $I(Y)$ is
a binomial ideal if and only if $V(I(Y))$ is a submonoid of
$\mathbb{A}^s$ under componentwise multiplication. This follows by adapting the proof of
Theorem~\ref{classification-binomial}.
\end{remark}


\begin{corollary}\label{oct11-14-1} If $\mathbb{Y}$ is a subset of
$\mathbb{P}^{s-1}$ which is closed in the Zariski topology, then
$I(\mathbb{Y})$ 
is a binomial ideal if and
only if $\mathbb{Y}\cup\{[0]\}$ is a submonoid of
$\mathbb{P}^{s-1}\cup\{[0]\}$.
\end{corollary}

\begin{proof} Thanks to Theorem~\ref{classification-binomial} 
it suffices to recall that $V(I(\mathbb{Y}))$ is equal to
$\overline{\mathbb{Y}}$ (see Lemma~\ref{mar31-14}). 
\end{proof}


\begin{corollary}\label{oct11-14-2} If $\mathbb{Y}$ is a subset of
$\mathbb{P}^{s-1}$ and $\dim(S/I(\mathbb{Y}))=1$, then  $I(\mathbb{Y})$ is a binomial ideal if and
only if $\mathbb{Y}\cup\{[0]\}$ is a submonoid of
$\mathbb{P}^{s-1}\cup\{[0]\}$.
\end{corollary}

\begin{proof} This is a direct consequence of Lemma~\ref{dim=1-finite} and
Corollary~\ref{oct11-14-1} because any finite set is closed in the
Zariski topology.
\end{proof}

\begin{definition}\label{projectivetorus-def} The set 
$T=\{[(x_1,\ldots,x_s)]\in\mathbb{P}^{s-1}\vert\, x_i\in
K^*\mbox{ for all }i\}$ is called a {\it projective torus} 
in $\mathbb{P}^{s-1}$, and the set $T^*=(K^*)^s$ is called an 
{\it affine torus\/} in $\mathbb{A}^s$, where $K^*=K\setminus\{0\}$.
\end{definition}

If $Y$ is a submonoid of an affine torus $T^*$, then $I(Y)$ is a
non-graded lattice ideal (see \cite[Proposition~2.3]{EisStu}). The
following corollary is the graded version of this result.

\begin{corollary}\label{oct11-14} 
If $\mathbb{Y}$ is a submonoid of a projective torus $T$, 
then $I(\mathbb{Y})$ is a lattice ideal.
\end{corollary}

\begin{proof} By Theorem~\ref{monoid-binomial}, $I(\mathbb{Y})$ is a
binomial ideal. Thus it suffices to show that $t_i$ is not a
zero-divisor of $S/I(\mathbb{Y})$ for all $i$. If $f\in S$ and $t_if$ 
vanishes at all points of $\mathbb{Y}$, then so does $f$, as 
required.
\end{proof}

\begin{corollary}{\cite[Proposition~6.7(a)]{lattice-dim1} }\label{lattice-classify} 
If $\mathbb{Y}\subset\mathbb{P}^{s-1}$ and $\dim(S/I(\mathbb{Y}))=1$,
then the following are equivalent{\rm:}
\begin{itemize}
\item[\rm(a)] $I(\mathbb{Y})$ is a lattice ideal.
\item[\rm(b)] $\mathbb{Y}$ is a finite subgroup of a projective torus
$T$.
\end{itemize}
\end{corollary}

\begin{proof} (a) $\Rightarrow$ (b): By  Lemma~\ref{dim=1-finite} the
set $\mathbb{Y}$ is finite. Using Corollary~\ref{oct11-14-2} and
Lemma~\ref{primdec-ixx} it 
follows that $\mathbb{Y}$ is a submonoid of $T$. As the cancellation laws hold in $T$ and
$\mathbb{Y}$ is finite, we get that $\mathbb{Y}$ is a group.

(b) $\Rightarrow$ (a): This is a direct consequence of
Corollary~\ref{oct11-14}. 
\end{proof}

\begin{proposition}\label{lattice-ideal-param-h} 
Let $K$ be an algebraically closed field. 
If $\mathbb{Y}\subset\mathbb{P}^{s-1}$, 
then the following are equivalent{\rm:}
\begin{itemize}
\item[\rm(a)] $\mathbb{Y}$ is a finite subgroup of a projective torus
$T$.
\item[\rm(b)] There is a finite subgroup $H$ of $K^*$ and 
$v_1,\ldots,v_s\in\mathbb{Z}^n$ such that 
$$\mathbb{Y}=\{[(x^{v_1},\ldots,x^{v_s})]\, \vert\, 
x=(x_1,\ldots,x_n)\mbox{ and }x_i\in
H\mbox{ for all }i\}\subset\mathbb{P}^{s-1}.
$$
\end{itemize}
\end{proposition}

\begin{proof} 

(b) $\Rightarrow$ (a): It is not hard to verify that $\mathbb{Y}$ is
a subgroup of $T$ using the parameterization of $\mathbb{Y}$ relative
to $H$. 

(a) $\Rightarrow$ (b): By the fundamental theorem of finitely generated
abelian groups, $\mathbb{Y}$ is a direct product of cyclic groups. Hence,
there are $[\alpha_1],\ldots,[\alpha_n]$ in $\mathbb{Y}$ such that
$$
\mathbb{Y}=\left\{\left.[\alpha_1]^{i_1}\cdots[\alpha_n]^{i_n}\, \right\vert\,
i_1,\ldots,i_n\in\mathbb{Z}\right\}.
$$
We set $\alpha_i=(\alpha_{i1},\ldots,\alpha_{is})$ for $i=1,\ldots,n$.
As $[\alpha_1],\ldots,[\alpha_n]$ have finite order, for each $1\leq
i\leq n$ there is $m_i={\rm o}([\alpha_i])$ such that
$[\alpha_i]^{m_i}=[\mathbf{1}]$. Thus 
$$
(\alpha_{i1}^{m_i},\ldots,\alpha_{is}^{m_i})=(\lambda_i,\ldots,\lambda_i)
$$
for some $\lambda_i\in K^*$. Pick $\mu_i\in K^*$ such that
$\mu_i^{m_i}=\lambda_i$. Setting, $\beta_{ij}=\alpha_{ij}/\mu_i$, one
has $\beta_{ij}^{m_i}=1$ for all $i,j$, that is all $\beta_{ij}$'s
are in $K^*$ and have finite order. Consider the subgroup $H$ of $K^*$
generated by all $\beta_{ij}$'s. This group is cyclic because $K$ is
a field. If $\beta$ is a generator of $(H,\cdot\, )$, we can write
$\alpha_{ij}/\mu_i=\beta^{v_{ji}}$ for some  $v_{ji}$ in $\mathbb{N}$.
Hence 
$$
[\alpha_1]=[(\beta^{v_{11}},\ldots,\beta^{v_{s1}})],\ldots,
[\alpha_n]=[(\beta^{v_{1n}},\ldots,\beta^{v_{sn}})].
$$ 

We set $v_i=(v_{i1},\ldots,v_{in})$ for $i=1,\ldots,s$. 
Let $\mathbb{Y}_H$ be the set in $\mathbb{P}^{s-1}$ parameterized by
the monomials $y^{v_1},\ldots,y^{v_s}$ relative to $H$. If
$[\gamma]\in\mathbb{Y}$, then we can write
$$
[\gamma]=[\alpha_1]^{i_1}\cdots[\alpha_n]^{i_n}=[(
(\beta^{i_1})^{v_{11}}\cdots (\beta^{i_n})^{v_{1n}},\ldots,
(\beta^{i_1})^{v_{s1}}\cdots (\beta^{i_n})^{v_{sn}}
)]
$$
for some $i_1,\ldots,i_n\in\mathbb{Z}$. Thus
$[\gamma]\in\mathbb{Y}_H$. Conversely if $[\gamma]\in\mathbb{Y}_H$,
then $[\gamma]=[(x^{v_1},\ldots,x^{v_s})]$ for some $x_1,\ldots,x_n$
in $H$. Since any $x_k$ is of the form $\beta^{i_k}$ for some integer
$i_k$, one can write $[\gamma]=[\alpha_1]^{i_1}\cdots[\alpha_n]^{i_n}$, that
is, $[\gamma]\in\mathbb{Y}$. 
\end{proof}

\begin{remark} The equivalence between
(a) and (b) was shown in \cite[Proposition~6.7(b)]{lattice-dim1}
under the assumption that $K$ is a finite field.
\end{remark}



\begin{proposition}\label{lattice-alg-closed-char=0} Let $K$ be an algebraically closed field of
characteristic zero and let $I$ be a graded ideal of $S$ of 
dimension $1$. Then $I$ is a lattice ideal if and only if $I$ is the
vanishing ideal of a finite subgroup $\mathbb{Y}$ of a projective
torus $T$.
\end{proposition}

\begin{proof} 
$\Rightarrow$) Assume that $I=I(\mathcal{L})$ is the lattice ideal of
a lattice $\mathcal{L}$ in $\mathbb{Z}^s$. Since $I$ is graded and
$\dim(S/I)=1$, for
each $i\geq 2$, there is $a_i\in\mathbb{N}_+$ such that 
$f_i:=t_i^{a_i}-t_1^{a_i}\in I$. This polynomial has a factorization
into linear factors of the form $t_i-\mu t_1$ with $\mu\in K^*$. In
characteristic zero a lattice ideal is radical \cite[Theorem~8.2.27]{monalg-rev}. Therefore $I$ is the
intersection of its minimal primes and each minimal prime is generated
by $s-1$ linear polynomials of the form $t_i-\mu t_1$. It follows 
that $I$ is the vanishing ideal of some finite subset $\mathbb{Y}$ of a
projective torus $T$. By Corollary~\ref{oct11-14-1}, $\mathbb{Y}$ is a
submonoid of $T$. As the cancellation laws hold in $T$ and
$\mathbb{Y}$ is finite, we get that $\mathbb{Y}$ is a group.

$\Leftarrow$) This implication follows at once from Corollary~\ref{oct11-14}. 
\end{proof}

\medskip

\noindent
{\bf Acknowledgments.} 
We thank Thomas Kahle for his comments and for pointing out the Remark 
after \cite[Proposition~2.3]{EisStu}.  The authors would also like to thank the referees for their
careful reading of the paper and for the improvements that
they suggested.

\bibliographystyle{plain}

\end{document}